\documentclass{article}
\usepackage{hyperref,amsmath,amsthm,amsfonts,etoolbox,amscd,flafter,epsf,amssymb,wasysym,graphicx,paralist,centernot}
\usepackage{authblk}

\usepackage[utf8]{inputenc}
\newtheoremstyle{mytheoremstyle} 
{\topsep}                    
    {\topsep}                    
    {}                   
    {}                           
    {\bf}                   
    {.}                          
    {.5em}                       
    {}  

\theoremstyle{mytheoremstyle}

\theoremstyle{definition}

\newcommand{\RR}{\mathbb{R}}
\newcommand{\NN}{\mathbb{N}}
\theoremstyle{remark}

\newcommand{\dt}{\ {\rm d} t }
\newcommand{\dint}{{\displaystyle \int}}
\theoremstyle{plain}
\newtheorem{theorem}{Theorem}[section]

\usepackage{url}
\makeatletter
\newcommand{\address}[1]{\gdef\@address{#1}}
\newcommand{\email}[1]{\gdef\@email{\url{#1}}}
\newcommand{\@endstuff}{\par\vspace{\baselineskip}\noindent\small
\begin{tabular}{@{}l}\scshape\@address\\\textit{E-mail address:} \@email\end{tabular}}
\AtEndDocument{\@endstuff}
\makeatother
\title{Diophantus Equations and Partially Ordered Sets}
\author{Addea Gupta}
\address{Sanskriti School, Chanakyapuri, New Delhi, 110021}
\email{addeagupta@gmail.com}
\usepackage{lipsum}
\date{\vspace{-5ex}}

\begin{document}

\maketitle
\begin{abstract}
In \cite{(k!)^n} it is shown that the Diophantine equation $(k!)^n+k^n=(n!)^k+n^k$ only has the trivial solution $n=k$, and  $(k!)^n-k^n=(n!)^k-n^k$ only has the solutions $n=k$, $(n, k)=(1, 2),$ and $(2, 1)$. In this article we find all solutions of the Diophantine Equations $a_1!a_2!\cdots a_n! \pm a_1a_2 \cdots a_n = b_1!b_2! \cdots b_k! \pm b_1b_2 \cdots b_k$, where $a_i$ majorizes $b_i$. Furthermore we find a sufficient condition on a function $f:\NN\to\RR^+$ to guarantee that $f$ gives a monotone function on the POSET of all finite sequences of natural numbers. We then use that to solve other Diophantine equations involving factorials and generalize the results of \cite{k^(k-1)}. We also explore similar Diophantine Equations for the Fibonacci Sequence and  other sequences of natural numbers given by linear recursions of the form $A_{n+2}=aA_{n+1}+bA_{n}$. \\ \\
{2010 \textit{Mathematics Subject Classification.} 06A06,11D72,11B65,33B15.}
\end{abstract}

\section{Introduction}

In \cite{(k!)^n} the authors prove if $(k!)^n+k^n=(n!)^k+n^k$, then $n=k$, and if $(k!)^n-k^n=(n!)^k-n^k$, then $n=k$ or $(n, k)=(1, 2), $ or $(2, 1)$. The idea of the proof is to use monotonicity of sequences $\sqrt[n]{n!}$ and $\sqrt[n]{n}$ to obtain the result. We generalize this result by first turning the set $S$ of all finite sequences of positive integers into a Partially Ordered Set using majorization. A sequence of positive integers $(a_1, a_2, \ldots, a_n)$ majorizes a sequence of positive integers $(b_1, b_2, \ldots, b_k)$ whenever all of the following holds:
\begin{itemize}
\item $n\leq k$, and 
\item For every $i\leq n$, $a_1+\cdots+a_i\geq b_1+\cdots+b_i$.
\item $a_1+\cdots+a_n\geq b_1+\cdots+b_k.$
\end{itemize}
In which case we write $(a_1, a_2, \ldots, a_n)\succ (b_1, b_2, \ldots, b_k).$\\

As a result we are able to solve similar yet more general Diophantine equations. For example  we prove that for finite sequences of positive integers $(a_1, \ldots, a_n)$ and $(b_1, \ldots, b_k)$ where $(a_1, \ldots, a_n)\succ (b_1, \ldots, b_k)$, then  $a_1!a_2!\cdots a_n! + a_1a_2 \cdots a_n = b_1!b_2! \cdots b_k! + b_1b_2 \cdots b_n$ implies $n=k$ and $a_j=b_j$ for all $j$. We also show that if $a_1!a_2!\cdots a_n! - a_1a_2 \cdots a_n = b_1!b_2! \cdots b_k! - b_1b_2 \cdots b_n$, then either $a_j=b_j$ for all $j$ or $a_j, b_j\in \{1, 2\}$ for all $j$. Setting $a_i=k$ and $b_i=n$, we obtain the main results proved in \cite{(k!)^n}.\\

We then find a sufficient condition on a function $f:\mathbb{Z}^+\to\mathbb{R}^+$ to impose a monotone function on $S$.  As our main result we prove the following theorem:\\

{\bf Theorem A.} Suppose $f:\mathbb{N}\longrightarrow \mathbb{R}^+ $ is a function satisfying:
\begin{itemize}
\item $f(0)=1$, and
\item  $\dfrac{f(x)}{f(x-1)}$ is strictly increasing (resp. decreasing)\end{itemize}
Then the following holds:\\
If $(a_1,\cdots,a_n) \succ (b_1,\cdots,b_k)$ for two sequences of positive integers, then $f(a_1)\cdots f(a_n) \geq f(b_1)\cdots f(b_k)$ (resp. $f(a_1)\cdots f(a_n) \leq f(b_1)\cdots f(b_k)$). Equality holds iff $k=n$ and $a_i=b_i$ for all $i$.

We apply the above theorem to appropriate functions to deduce some results of  \cite{k^(k-1)}. For instance we prove that the only solutions to all of the following Diophantine equations $$\begin{array}{c}(k!)^n n^{kn}=(n!)^k k^{kn},\\ \\ \left( {\dfrac {k^{(k-1)}}{(k-1)!}}\right)^{n(n-1)} = \left( {\dfrac {n^{(n-1)}}{(n-1)!}}\right)^{k(k-1)}, \text{and} \\
\\
\left(\dfrac{k^{k^2-1}}{(k-1)!}\right)^{n(n-1)}=\left(\dfrac{n^{n^2-1}}{(n-1)!}\right)^{k(k-1)}\end{array}$$ are $k=n$.

We will then prove a theorem similar to Theorem A for sums as follows.\\

{\bf Theorem B.} Suppose $f:\mathbb{N}\longrightarrow \mathbb{R}$ is a function satisfying:
\begin{itemize}
\item $f(0)=0$, and
\item  $f(x)-f(x-1)$ is strictly increasing (resp. decreasing).\end{itemize}
Then the following holds:\\
If $(a_1,\cdots,a_n) \succ (b_1,\cdots,b_k)$ for two sequences of positive integers, then $f(a_1)+\cdots+f(a_n) \geq f(b_1)+\cdots+f(b_k)$ (resp. $f(a_1)+\cdots+f(a_n) \leq f(b_1)+\cdots+f(b_k)$). Equality holds iff $k=n$ and $a_i=b_i$ for all $i$.

This theorem is then used to generalize two of the other results of \cite{k^(k-1)}  as follows:

Suppose $(a_1, a_2,\ldots, a_n)\succ (b_1, b_2, \ldots, b_k)$, then the only solutions to the following Diophantine equations are $a_i=b_i$, and $k=n$.
$$\begin{array}{c}\sum\limits_{i=1}^n ((a_i+1)!)^{1/(a_i+2)}= \sum\limits_{i=1}^k ((b_i+1)!)^{1/(b_i+2)}\\ \sum\limits_{i=1}^n ((a_i+2)!)^{1/(a_i+2)}= \sum\limits_{i=1}^k ((b_i+2)!)^{1/(b_i+2)}  \end{array}$$

Recall that the Gamma function is defined by $\Gamma(x)=\dint_0^\infty e^{-t}t^{x-1}\dt$. It is well-known that for every $x>0$ we have $\Gamma(x+1)=x\Gamma(x)$. We will also make use of the following properties of the Gamma function that follow from Lemma 3 of \cite{k^(k-1)}:\\ $\ln(x-1)<\psi(x)<\ln(x)$, and $x\ln x-x+1<\ln\Gamma(x+1)<(x+1)\ln(x+1)-x$, where $\psi(x)$ denotes Euler's Digamma Function and $\psi(x)=\Gamma'(x)/\Gamma(x)$\\

{\bf Theorem C.} For every real number $x>1$ we have the following:
\begin{itemize}
    \item $\ln (\Gamma(x))>(x-0.5)\ln x -x$.
    \item $\dfrac{\Gamma'(x)}{\Gamma(x)}<\ln x.$
\end{itemize}

{\bf Definition.} A sequence of positive integers $a_1\cdots a_n$ is said to satisfy the uniqueness property if the only solution to $a_{n_1}\cdots a_{n_k}= a_{m_1}\cdots a_{m_l}$, where $n_1>\cdots >n_k$, and $m_1>\cdots >m_\ell$ and $(n_1,\ldots , n_k) \succ (m_1,\ldots ,m_\ell)$ then $k=\ell$ and $m_i=n_i$ for all $i$.\\

Throughout this paper, $F_n$ denotes the Fibonacci sequence defined recursively by $F_0=1, F_1=1,$ and $F_{n+2}=F_{n+1}+F_n$ for all $n\geq 0.$ It is well-known that $F_n= \dfrac{\alpha^{n+1}-\beta^{n+1}}{\alpha-\beta}$, where $\alpha,\beta$ are roots of $x^2-x-1=0$.

\section{Main Results}

\begin{theorem}\label{majorizing} Let $a_1, \ldots, a_n$ and $b_1,\ldots, b_k$ be sequences of positive integers for which  $(a_{1},...,a_{n})\succ(b_{1},...,b_{k})$. Then $a_{1}!\cdots a_{n}!\geq b_1!\cdots b_k!$. Furthermore, equality holds if and only if $n=k$ and $a_j=b_j$ for all $j$.
\end{theorem}

\begin{proof} We will prove the statement by induction on $\sum\limits_{j=1}^{n} a_{j} + \sum\limits_{j=1}^k b_j$.

{\it Basis step}: If $\sum\limits_{j=1}^n a_j + \sum\limits_{j=1}^{k} b_{j} = 2$, then $n=k=1$ and $a_1=b_1=1$, and the claim clearly holds. \\

{\it Inductive step}: Suppose $(a_1,...,a_n)\succ(b_1,...,b_k)$. We will consider two cases:\\ 

{\it Case 1}: There is $1\leq i < n$ such that $a_{1}+\cdots+a_{i} = b_{1}+\cdots+b_i$. By assumption $(a_{1},\ldots,a_{i})\succ(b_{1},\ldots,b_i)$. Thus, by inductive hypothesis $$a_{1}!\cdots a_{i}! \geq b_{1}!\cdots b_{i}!\; (*).$$ 

We claim that $(a_{i+1},...,a_{n})\succ(b_{i+1},...,b_{k})$. \\

Note that since  $(a_{1},...,a_{n})\succ(b_{1},...,b_{k}),$ for every $i<\ell\leq n$ we have $\sum\limits_{j=1}^\ell a_j \geq \sum\limits_{j=1}^\ell b_j$.  Since $a_{1}+\cdots+a_{i} = b_{1}+\cdots+b_i$, we obtain $\sum\limits_{j=i+1}^\ell a_j \geq \sum\limits_{j=i+1}^\ell b_j$. This completes the proof of the claim.\\

By inductive hypothesis $a_{i+1}!\cdots a_n! \geq b_{i+1}!\cdots b_k!$. Multiplying this with $(*)$ we obtain  the result.\\ 

Now, suppose $a_1!\cdots a_n!=b_1!\cdots b_k!$. By what we proved above we must have $a_{1}!\cdots a_i! = b_{1}!\cdots b_{i}!$ and $a_{i+1}!\cdots a_{n}! = b_{i+1}!\cdots b_{k}!$ By inductive hypothesis $k=n$ and $a_j=b_j$ for all $j$.\\

{\it Case 2}: $a_{1}+\cdots +a_{i} > b_{1}+\cdots +b_{i}$ for all $i$ with $1\leq i<n.$ Assume $a_1=\cdots=a_j > a_{j+1} \geq \cdots \geq a_{n}$. Note that if $a_1=\cdots=a_n$ then we set $j=n$.\\ 

If $b_k>1$ then, $(a_1,\cdots,a_{j-1},a_j-1,a_{j+1}\cdots a_{n}) \succ 
(b_1,\ldots b_{k-1},b_{k}-1)$.\\ By inductive hypothesis $$a_1!\cdots a_{j-1}!(a_{j}-1)!\cdots a_n! \geq b_{1}!\cdots b_{k-1}!(b_{k}-1)!$$ Since $a_{1}=a_{j}\geq b_{1}\geq b_{k}$, we have $a_{1}!\cdots a_{n}! \geq b_1!\cdots b_{k}!$. If the equality holds, then we must have $a_1=b_1$. By assumption of this case we must have $n=1$. However, since $a_1\geq b_1+\cdots+b_k,$ we must have $k=1$ as well, and thus $n=k=1$ and $a_1=b_1,$ as desired.\\

Now suppose $b_k=1$. We see that  $(a_1\cdots a_{j-1},a_j-1,a_{j+1}\cdots a_n) \succ 
(b_{1}\cdots b_{k-1})$. Thus, $a_{1}!\cdots a_{j-1}!(a_{j}-1)!\cdots a_{n}! \geq b_1!\cdots b_{k-1}!$. Since $a_{1}=a_{j} \geq b_{1} \geq b_{k}=b_{k}!$ multiplying the two inequalities yields $a_1! \cdots a_n! \geq b_1!\cdots b_k!$. If the equality holds, we must have $a_1=b_1$, and thus $n=1$. The rest is similar to when $b_k>1$.
\end{proof}

\begin{theorem}\label{a!+a}
Suppose $(a_1,\ldots,a_n)\succ (b_1,\ldots,b_k)$ where $a_i,b_i$ are decreasing sequences of positive integers.
\\
i) If $a_1!a_2!\cdots a_n! + a_1a_2 \cdots a_n = b_1!b_2! \cdots b_k! + b_1b_2 \cdots b_k$ then $n=k$ and $a_i=b_i$ for all i
\\
ii)If $a_1!a_2!\cdots a_n! - a_1a_2 \cdots a_n = b_1!b_2! \cdots b_k! - b_1 b_2 \cdots b_k$ then either (a) $n=k$ and  $a_i=b_i$ for all i or (b) $a_i, b_i \in  \{1, 2\}$
\end{theorem}

\begin{proof}
(i) Assume $a_m>b_m$ and $a_1=b_1,\ldots,a_{m-1}=b_{m-1}$.
\\
$a_1\cdots a_n((a_1-1)!\cdots (a_n-1)!+1)=b_1\cdots b_k((b_1-1)!\cdots(b_k-1)!+1)$. Therefore, 
$$a_m\cdots a_n((a_1-1)!\cdots (a_n-1)!+1)=b_m\cdots b_k((b_1-1)!\cdots(b_k-1)!+1)$$\\
Since $b_j\leq a_m-1$ for all $j$ with $m\leq j \leq k$, we have $b_j\mid (a_1-1)! \cdots (a_n-1)!$. This implies $\gcd(b_j, (a_1-1)! \cdots (a_n-1)!+1)=1$ and $\gcd(b_m\cdots b_k, (a_1-1)! \cdots (a_n-1)!+1)=1$ Thus, $b_m\cdots b_k\mid a_m\cdots a_n \implies b_m\cdots b_k \leq a_m\cdots a_n $ and thus $$b_1\cdots b_k \leq a_1\cdots a_n\; (*)$$
We know $(a_1,\cdots ,a_n)\succ(b_1,\cdots,b_k)$. By Theorem~\ref{majorizing} $a_1!\cdots a_n!\geq b_1!\cdots b_k!$. Combining this with $(*)$ we obtain $$a_1!\cdots a_n!+a_1\cdots a_n\geq b_1!\cdots b_k!+b_1\cdots b_k.$$ Since equality holds we must have $a_1!\cdots a_n!= b_1!\cdots b_k!$ Therefore, by Theorem~\ref{majorizing} we have $n=k$ and $a_j=b_j$ for all $j$. \\

(ii) Note that $n!\geq n$ for all positive integers $n$. Thus, $a_1!\cdots a_n!-a_1\cdots a_n \geq 0$\\

{\it Case 1}: Suppose $a_1!\cdots a_n!- a_1\cdots a_n=0$. Therefore, $a_1!\cdots a_n!= a_1\cdots a_n $ Thus, $(a_1-1)!\cdots (a_n-1)!=1$. Therefore, $a_i-1=0,1$ and $a_i=1,2$ for all $i$. Similarly $b_i=1,2$ for all $i$ which gives us (b).\\

{\it Case 2}: Suppose $a_1!\cdots a_n!-a_1\cdots a_n>0$. By a similar argument to (i) we deduce  $a_1\cdots a_n\geq b_1\cdots b_k$. \\

Note that $(a_1,\cdots, a_n) \succ (b_1,\cdots,b_k)$ implies $a_1+\cdots+a_n\geq b_1+\cdots+b_k$ and $n\geq k$. Therefore, $(a_1-1)+\cdots+(a_n-1)\geq (b_1-1)+\cdots+(b_k-1)$.\\ 
We can say $((a_1-1),\cdots,( a_n-1)) \succ ((b_1-1),\cdots,(b_k-1), 1,\cdots,1)$ .\\
By $(*)$ we have $(a_1-1)!\cdots (a_n-1)!\geq (b_1-1)!\cdots (b_k-1)! 1!\cdots 1!$.\\
Since equality holds, $a_i-1=b_i-1$ for all $i$.Hence, $a_i=b_i$ for all $i$ and by $(*)$ $n=k$
\end{proof}

\begin{theorem}\label{main}
Suppose $f:\mathbb{N}\longrightarrow \mathbb{R}^+ $ is a function satisfying:
\begin{itemize} 
\item $f(0)=1$, and
\item  $\dfrac{f(x)}{f(x-1)}$ is strictly increasing (resp. decreasing).
\end{itemize}
Then the following holds:\\

If $(a_1,\cdots,a_n) \succ (b_1,\cdots,b_k)$ for two sequences of positive integers, then $f(a_1)\cdots f(a_n) \geq f(b_1)\cdots f(b_k)$ (resp. $f(a_1)\cdots f(a_n) \leq f(b_1)\cdots f(b_k)$). Furthermore, equality holds if and only if $k=n$ and $a_i=b_i$ for all $i$.
\end{theorem}

\begin{proof}
We will prove this by strong induction on $\sum\limits_{i=1}^n a_i+\sum\limits_{i=1}^k b_i$.\\ 
{\it Basis Step:} If $\sum\limits_{i=1}^n a_i+\sum\limits_{i=1}^k b_i=2$, then $a_1=b_1=1$, and the result is clear.\\
{\it Inductive Step:} Similar to the proof of Theorem~\ref{majorizing} we will consider two cases:\\
{\it Case 1}: For some $j$ with $1\leq j< n$, we have $a_1+\cdots+a_j=b_1+\cdots+b_j$. Thus $(a_1,\cdots,a_j)\succ(b_1,\cdots,b_j)$ and $(a_{j+1},\cdots,a_n)\succ(b_{j+1},\cdots,b_k)$. Therefore, by inductive hypothesis $\prod_{i = 1}^{j} f(a_{i})\geq \prod_{i = 1}^{j} f(b_{i})$ and $\prod_{i = j+1}^{n} f(a_{i})\geq \prod_{i = j+1}^{k} f(b_{i})$. Multiplying these two we obtain the result. By inductive hypothesis the equality holds if and only if $k=n$ and $a_i=b_i$ for all $i$.\\

{\it Case 2}: For all $j<n$, we have $a_1+\cdots+a_j>b_1+\cdots+b_j$. Suppose $a_1=\cdots=a_j>a_{j+1}\geq \cdots \geq a_n$\\
If $a_n=b_k=1$, then $(a_1,\cdots,a_{n-1})\succ(b_1,\cdots,b_{k-1})$. The rest follows from the inductive hypothesis.\\
If $b_k=1,$ and $a_n>1$, then $(a_1,\cdots,a_n-1)\succ(b_1,\cdots,b_{k-1})$. This implies $\prod\limits_{i = 1}^{n-1} f(a_{i})f(a_n-1)\geq\prod\limits_{i = 1}^{k-1} f(b_{i})$.  By assumption $\dfrac{f(a_n)}{f(a_n-1)} >\dfrac{f(1)}{f(0)}=f(b_k).$ Multiplying the two inequalities we obtain the result. \\
If $b_k>1$, then $(a_1, \ldots, a_{j-1}, a_j-1, a_{j+1}, \ldots, a_n)\succ (b_1, \ldots, b_{k-1}, b_k-1)$. By inductive hypothesis $f(a_1)\cdots f(a_{j-1})f(a_j-1)f(a_{j+1})\cdots f(a_n)\geq f(b_1)\cdots f(b_{k-1})f(b_k-1).$ By assumption $a_j=a_1\geq b_1\geq b_k$, and thus $\dfrac{f(a_j)}{f(a_j-1)}\geq \dfrac{f(b_k)}{f(b_{k-1})}$. Multiplying these inequalities we obtain the result. 

If the equality holds, then we must have $a_j=b_k$, which means $a_1=b_1$, which by assumptions of this case we conclude that $n=1$ and since $a_1\geq b_1+\cdots+b_k$, we must have $n=k=1$ and this concludes the proof of the equality case.
\end{proof}

\begin{theorem}\label{main2}
Suppose $f:\mathbb{N}\longrightarrow \mathbb{R}$ is a function satisfying:
\begin{itemize} 
\item $f(0)=0$, and
\item  $f(x)-f(x-1)$ is strictly increasing (resp. decreasing).
\end{itemize}
Then the following holds:\\

If $(a_1,\cdots,a_n) \succ (b_1,\cdots,b_k)$ for two sequences of positive integers, then $f(a_1)+\cdots+f(a_n) \geq f(b_1)+\cdots+f(b_k)$ (resp. $f(a_1)+\cdots+f(a_n) \leq f(b_1)+\cdots+f(b_k)$). Furthermore, equality holds if and only if $k=n$ and $a_i=b_i$ for all $i$.
\end{theorem}

\begin{proof}
Suppose $f(x)-f(x-1)$ is strictly increasing, and consider the function $g(x)=e^{f(x)}$, and note that $g(x)$ satisfies the conditions of Theorem~\ref{main}. Therefore, $\prod g(a_i)\geq \prod g(b_i)$ and thus $\sum f(a_i)\geq \sum f(b_i)$. Furthermore, the equality case follows from Theorem~\ref{main}.
\end{proof}

\section{Applications}

Suppose $k\geq n$ are two positive integers. Then, $(\underbrace{k, \ldots, k}_{n\text{ times}})\succ(\underbrace{n, \ldots, n}_{k\text{ times}}) $. Using these two sequences in Theorem~\ref{a!+a} we obtain the following that is the main result of \cite{(k!)^n}.

\begin{theorem}  Let $n$ and $k$ be positive integers. Then, 
\begin{itemize}
\item $(k!)^n + k^n = (n!)^k +n^k$ holds if and only if
$k = n$.
\item $(k!)^n-k^n = (n!)^k-n^k$ holds if and only if $k = n$ or $(k, n) = (1, 2), (2, 1)$.
\end{itemize}
\end{theorem}

\begin{theorem}\label{Thm 3.2}
Let $a_i$ and $b_i$ be two sequences of positive integers such that $(a_1, a_2, \ldots, a_n)\succ(b_1, b_2, \ldots, b_k)$. Then the following equations only have the trivial solutions $n=k$, and $a_i=b_i$ for all $i$.
$$\begin{array}{c}\dfrac{a_1!a_2!\cdots a_n!}{a_1^{a_1}a_2^{a_2}\cdots a_n^{a_n}} =\dfrac{b_1!b_2!\cdots b_k!}{b_1^{b_1}b_2^{b_2}\cdots b_k^{b_k}} \qquad (1)\\ 
\\
\prod\limits_{i=1}^n\dfrac{a_i}{((a_i-1)!)^{\frac{1}{a_i-1}}} =\prod\limits_{i=1}^k\dfrac{b_i}{((b_i-1)!)^{\frac{1}{b_i-1}}} \qquad (2)\\
\\
\prod\limits_{i=1}^n\dfrac{a_i^{a_i+1}}{((a_i-1)!)^{\frac{1}{a_i-1}}} =\prod\limits_{i=1}^k\dfrac{b_i^{b_i+1}}{((b_i-1)!)^{\frac{1}{b_i-1}}} \qquad (3)\\ \end{array} $$
\end{theorem}

\begin{proof}
It is enough to prove the following functions satisfy the conditions of Theorem~\ref{main}.
\begin{enumerate}[i.]
\item $f(0)=1$, and $f(x)=\dfrac{x!}{x^x}$, when $x\geq 1$.
\item $f(0)=1$, $f(1)=1.5$, $f(x)=\dfrac{x}{((x-1)!)^{1/(x-1)}}$, when $x\geq 2$.
\item $f(0)=1$, $f(1)=2$, $f(x)=\dfrac{x^{x+1}}{((x-1)!)^{1/(x-1)}}$, when $x\geq 2$.
\end{enumerate}
(i) Let $g(x)= \dfrac{f(x)}{f(x-1)}$ for all positive integers $x$. We note that  $g(1)=1$, and  $g(x)=\left(\dfrac{x-1}{x}\right)^{x-1}$ for all $x>1$. We will show $\ln(g(x))$ is strictly decreasing.\\

Let $h(x)=\ln(g(x))=(x-1)\ln\left(1-\dfrac{1}{x}\right)$.\\
$h'(x)= \ln\left(1-\dfrac{1}{x}\right) + \dfrac{1}{x}$, and thus $h''(x)=\dfrac{1}{x^2(x-1)}>0$ for all $x>1$. This means $h'(x)$ is strictly increasing over $[2, \infty)$. This implies $$h'(x)<\lim\limits_{t\to\infty} h'(t)= \ln (1-0)+0=0,$$ which means $h$ is strictly decreasing over $[2, \infty)$. On the other hand, $h(1)=0 > h(2)=\ln(0.5)$, and hence $h(x)$ is strictly decreasing over $[1, \infty)$. This completes the proof.\\

(ii) To prove $\dfrac{f(x)}{f(x-1)}$ is strictly decreasing, we will need to show the derivative of  $\ln \dfrac{f(x)}{f(x+1)}=\ln f(x)-\ln f(x+1)$ is positive. Using the fact that $\Gamma(x)=(x-1)!$ we have $$\ln f(x)= \ln x-\dfrac{\ln\Gamma(x)}{x-1}.$$ Also note that $\Gamma(x+1)=x\Gamma(x),$ and thus $$\ln f(x+1)= \ln (x+1)-\dfrac{\ln x+\ln\Gamma(x)}{x}.$$
Note that the derivative of $\ln f(x)-\ln f(x+1)$ is equal to:
$$\begin{array}{cc}
& \dfrac{1}{x}-\dfrac{\Gamma'(x)}{\Gamma(x)(x-1)}+\dfrac{\ln \Gamma(x)}{(x-1)^2}-\dfrac{1}{x+1}+\dfrac{1}{x^2}-\dfrac{\ln x}{x^2}-\dfrac{\ln \Gamma(x)}{x^2}+\dfrac{\Gamma'(x)}{x\Gamma(x)} \\ \\ = & \dfrac{1}{x(x+1)}+ \dfrac{\ln\Gamma(x)(2x-1)}{x^2(x-1)^2}-\dfrac{\Gamma'(x)}{\Gamma(x)x(x-1)}-\dfrac{\ln x}{x^2}+\dfrac{1}{x^2}\\ \\ > & \dfrac{1}{x(x+1)}+ ((x-1/2)\ln x-x)\dfrac{2x-1}{x^2(x-1)^2}-\dfrac{\ln x}{x(x-1)}-\dfrac{\ln x}{x^2}+\dfrac{1}{x^2} \\ \\ = & \ln x\left(\dfrac{(2x-1)^2}{2x^2(x-1)^2}-\dfrac{1}{x(x-1)}-\dfrac{1}{x^2}\right) + \dfrac{1}{x(x+1)}- \dfrac{x(2x-1)}{x^2(x-1)^2}+\dfrac{1}{x^2}\\ \end{array}$$ 
Here we used the inequalities in Theorem C.
On combining the log terms and the fractions we get, 
$$\begin{array}{cc} & \ln x\left(\dfrac{2x-1}{2x^2(x-1)^2}\right) + \dfrac{-4x^2+x+1}{x^2(x-1)^2(x+1)}\\ \\ = & \dfrac{(2x^2+x-1)\ln x-8x^2+2x+2}{2x^2(x-1)^2(x+1)} \\ \\ > & \dfrac{4(2x^2+x-1)-8x^2+2x+2}{2x^2(x-1)^2(x+1)} \\ \\ = & \dfrac{6x-2}{2x^2(x-1)^2(x+1)}>0,\end{array}$$
assuming $\ln x > 4$. When $x\leq e^4$, we can see that $\dfrac{f(x)}{f(x-1)}$ is strictly decreasing.\\

\begin{center}
    \includegraphics[scale= 0.5]{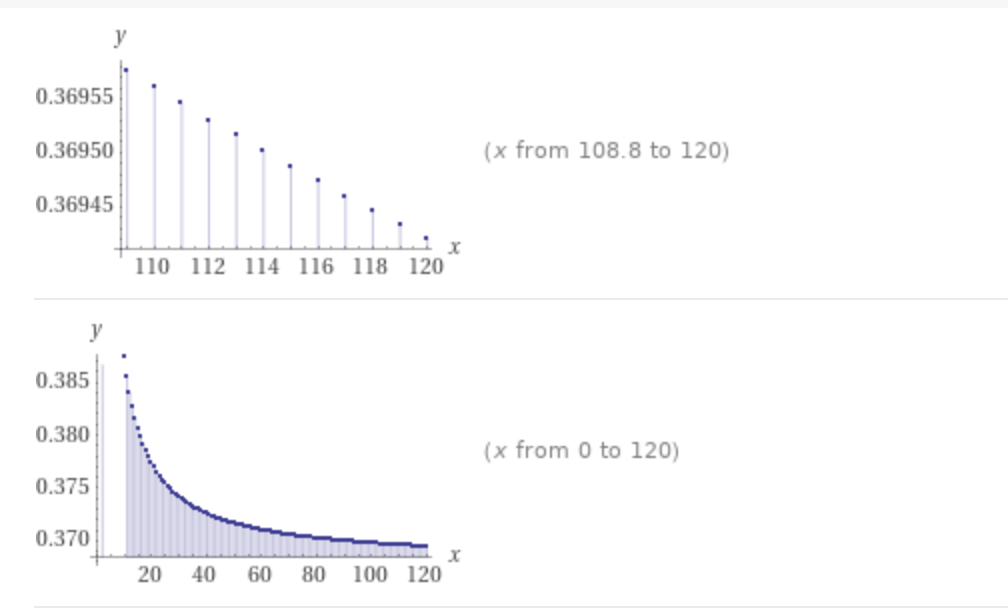}
\end{center} 

(iii) We note that $\dfrac{f(x)}{f(x-1)}= \dfrac{x^{x+1}}{(x-1)^x}\dfrac{(x-2)!^{1/(x-2)}}{(x-1)!^{1/(x-1)}}$ for all $x\geq 3$.\\ 

Let $h(x)=\ln\left(\dfrac{x^{x+1}}{(x-1)^x}\right)= (x+1)\ln x -x\ln (x-1)$. Then, $$h'(x)=\ln x+ \dfrac{x+1}{x} -\dfrac{x}{x-1}-\ln (x-1)= \left(\ln x +\dfrac{1}{x}\right) - \left(\ln(x-1) + \dfrac{1}{x-1}\right)$$

Let $k(x)= \ln x +\dfrac{1}{x}$. We have $k'(x)= \dfrac{1}{x}-\dfrac{1}{x^2}>0$. Therefore, $k(x)$ is strictly increasing. Thus, $k(x)> k(x-1)$, for all $x$, and thus $h'(x)>0$, which implies $h(x)$ is strictly increasing.\\
We will now prove $\dfrac{((x-2)!)^{1/(x-2)}}{((x-1)!)^{1/(x-1)}}<\dfrac{((x-1)!)^{1/(x-1)}}{(x!)^{1/x}}$. Clearing the denominator in the exponents and simplifying we get $$\left(\dfrac{(x-1)^{x+1}}{x^{x-1}}\right)^{x-2}>((x-2)!)^2.$$

$(x-2)!^2$ can be written as the product of $(k+1)(x-2-k),$ where $0\leq k\leq x-3$. By AM-GM inequality, $$(k+1)(x-2-k)\leq \left(\dfrac{k+1+x-2-k}{2}\right)^2= \left(\dfrac{x-1}{2}\right)^2$$ Thus, it is enough to prove $\dfrac{(x-1)^{x+1}}{x^{x-1}}> \left(\dfrac{x-1}{2}\right)^2$. This is equivalent to  $\dfrac{(x-1)^{x-1}}{x^{x-1}}>\dfrac{1}{4}$. We will show that $\dfrac{x^{x-1}}{(x-1)^{x-1}}<4$ for all $x\geq 3$. 
Let $g(x)=  (x-1)\ln x-(x-1)\ln (x-1)$. Then, $$g'(x)=\dfrac{x-1}{x} + \ln x -\dfrac{x-1}{x-1}- \ln (x-1)= \ln x -\dfrac{1}{x} -\ln (x-1).$$ 
So, $$g''(x)= \dfrac{1}{x} +\dfrac{1}{x^2}-\dfrac{1}{x-1}= -\dfrac{1}{x(x-1)} +1/x^2<0.$$ Hence, $g'(x)$ is decreasing. 
$$g'(x)= \ln (x/(x-1)) -1/x > \lim\limits_{x\to\infty} g'(x)=\ln 1 -0=0.$$ Therefore, $g(x)=\dfrac{(x-1)^{x-1}}{x^{x-1}}$ is increasing. This implies
$$\dfrac{(x-1)^{x-1}}{x^{x-1}}< \lim\limits_{x\to\infty} \left(\dfrac{x}{x-1}\right)^{x-1}= \lim_{x\to\infty} (1+1/(x-1))^{x-1}=e<4$$
Note that $\dfrac{f(3)}{f(2)}=\dfrac{81}{8\sqrt2}, \dfrac{f(2)}{f(1)}=4$, and $\dfrac{f(1)}{f(0)}=2$, which implies $\dfrac{f(1)}{f(0)}<\dfrac{f(2)}{f(1)}<\dfrac{f(3)}{f(2)}$. Therefore, $\dfrac{f(x)}{f(x-1)}$ is strictly increasing, which means $f$ satisfies the conditions of Theorem~\ref{main}. This completes the proof.
\end{proof}

Setting $a_i=k$, and $b_i=n$ in Theorem~\ref{Thm 3.2} we obtain the following which are the main results of \cite{k^(k-1)}.

\begin{theorem}
Let $n$ and $k$ be two positive integers. Then the only solution to each of the following Diophantine equations is $n=k$.
\begin{itemize}
\item[(i)] $(k!)^nn^{nk}=(n!)^{k}k^{nk}$.
\item[(ii)] $\left(\dfrac{k^{k-1}}{(k-1)!}\right)^{n(n-1)}= \left(\dfrac{n^{n-1}}{(n-1)!}\right)^{k(k-1)}$
\item[(iii)] $\left(\dfrac{k^{k^2-1}}{(k-1)!}\right)^{n(n-1)}=\left(\dfrac{n^{n^2-1}}{(n-1)!}\right)^{k(k-1)} $.

\end{itemize}
\end{theorem}

\begin{proof}
Let $a_i=k$ and $b_i=n$ in the equations in Theorem~\ref{Thm 3.2} we get:\\
(i) $\dfrac{(k!)^n}{k^{nk}}= \dfrac{(n!)^k}{n^{nk}} $. Cross multiplying, we get the result $(k!)^n n^{nk}= (n!)^k k^{nk}$\\ \\
(ii) $\left(\frac{k}{(k-1)!^{\frac{1}{k-1}}}\right)^n= \left(\frac{n}{(n-1)!^{\frac{1}{n-1}}}\right)^k$. Clearing the denominator in the exponents we get the result $\left(\dfrac{k^{k-1}}{(k-1)!}\right)^{n(n-1)}= \left(\dfrac{n^{n-1}}{(n-1)!}\right)^{k(k-1)}$\\ \\
(iii)$\left(\frac{k^{k+1}}{(k-1)!^{\frac{1}{k-1}}}\right)^n= \left(\frac{n^{n+1}}{(n-1)!^{\frac{1}{n-1}}}\right)^k$.Clearing the denominator in the exponents we get the result $\left(\dfrac{k^{k^2-1}}{(k-1)!}\right)^{n(n-1)}=\left(\dfrac{n^{n^2-1}}{(n-1)!}\right)^{k(k-1)} $\\ 
\end{proof}

\begin{theorem}
Let $a_i$ and $b_i$ be two sequences of positive integers such that $(a_1, a_2, \ldots, a_n)\succ(b_1, b_2, \ldots, b_k)$. Then the following equations only have the trivial solutions $n=k$, and $a_i=b_i$ for all $i$.
$$\begin{array}{c}\sum\limits_{i=1}^n ((a_i+1)!)^{1/(a_i+2)}= \sum\limits_{i=1}^k ((b_i+1)!)^{1/(b_i+2)}\\
\\ 
\sum\limits_{i=1}^n ((a_i+2)!)^{1/(a_i+2)}= \sum\limits_{i=1}^k ((b_i+2)!)^{1/(b_i+2)}\end{array}$$
\end{theorem}

\begin{proof}
It is enough to prove the following functions satisfy the conditions of Theorem~\ref{main2}.
\begin{enumerate}[(i)]
    \item $f_1(x)=((x+1)!)^{1/(x+2)}.$
    \item $f_2(x)=((x+2)!)^{1/(x+2)}.$
\end{enumerate}
Note that $f_1(x)=(\Gamma(x+2))^{1/(x+2)}$, and $f_2(x)=(\Gamma(x+3))^{1/(x+2)}$. In \cite{k^(k-1)} it is shown that $f_1(x+1)-f_1(x)$ and $f_2(x+1)-f_2(x)$ are both strictly monotone, which means $f_1$ and $f_2$ satisfy the properties of Theorem~\ref{main2}, as desired.
\end{proof}

\begin{theorem}\label{fibonacci} Let $F_n$ be the Fibonacci sequence. Suppose $F_{2n_1}\cdots F_{2n_k}= F_{2m_1}\cdots F_{2m_\ell}$, where $n_1>\cdots >n_k$ and $m_1>\cdots >m_\ell$ and $(n_1,\cdots ,n_k) \succ (m_1,\cdots ,m_\ell)$, then $k=\ell$ and $m_i=n_i$ for all $i$.
\end{theorem}

\begin{proof} We will use the fact that $F_m= \dfrac{\alpha^{m+1} - \beta^{m+1}}{\alpha-\beta}$, where $\alpha,\beta$ are roots of $x^2-x-1=0$. By Theorem~\ref{main}, it is enough to prove  $\dfrac{F_{2n+2}}{F_{2n}}>\dfrac{F_{2n}}{F_{2n-2}}$.\\

This is equivalent to $F_{2n+2}F_{2n-2}>F_{2n}^2$, which is equivalent to $$(\alpha^{2n+3}-\beta^{2n+3})(\alpha^{2n-1}-\beta^{2n-1})>(\alpha^{2n+1}-\beta^{2n+1})^2$$ Simplifying we obtain $-\alpha^{2n+3}\beta^{2n-1}-\beta^{2n+3}\alpha^{2n-1}>-2\alpha^{2n+1}\beta^{2n+1}$. Dividing by $\alpha^{2n-1}\beta^{2n-1}=-1$ we get $-\alpha^4-\beta^4<-2\alpha^2\beta^2$ which is equivalent to $(\alpha^2-\beta^2)^2>0$. This completed the proof.
\end{proof}

\begin{theorem}
If $F_{2n_1+1}\cdots F_{2n_k+1}= F_{2m_1+1}\cdots F_{2m_\ell+1}$ and $n_1>\cdots >n_k$ and $m_1>\cdots >m_\ell$ and $(n_1,\cdots, n_k) \succ (m_1,\cdots, m_\ell)$ then $k=\ell$ and $m_i=n_i$ for all $i$.
\end{theorem}

\begin{proof}
The proof is similar to that of Theorem~\ref{fibonacci}
\end{proof}

\begin{theorem}
Let $a$ be a positive integer and $b$ be a negative integer and $A_n$ be a sequence of non-negative integers satisfying $A_0=1$, $A_1^2-aA_1-b>0$, and $A_{n+2}=aA_{n+1}+bA_{n}$ for all $n\geq 0$. Then the sequence $A_n$ satisfies the uniqueness property.
\end{theorem}

\begin{proof} By Theorem~\ref{main} it is enough to prove $\dfrac{A_{n+1}}{A_{n}}$ is strictly monotone. Note that $$\dfrac{A_{n+1}}{A_n}-\dfrac{A_{n+2}}{A_{n+1}}=\dfrac{A_{n+1}}{A_n}-\dfrac{aA_{n+1}+bA_n}{A_{n+1}}=\dfrac{A_n}{A_{n+1}}\left(\left(\dfrac{A_{n+1}}{A_n}\right)^2- a\dfrac{A_{n+1}}{A_n}-b\right)$$ 

Letting $q(x)= x^2-ax-b$, we need to prove that either for all $n$, $q\left(\dfrac{A_{n+1}}{A_n}\right)>0$ or  for all $n$, $q\left(\dfrac{A_{n+1}}{A_{n}}\right)<0$.\\

If the quadratic $q(x)$ has no real roots, then it is always positive, which completes the proof.\\

Assume $\alpha<\beta$ are roots of $q(x)=0$. Note that since $\alpha+\beta=a$ and $\alpha\beta=-b$ are both positive, $\alpha$ and $\beta$ are positive. In order to have $q\left(\dfrac{A_{n+1}}{A_n}\right)>0$, we need $\dfrac{A_{n+1}}{A_n}>\beta$ or $\dfrac{A_{n+1}}{A_n}<\alpha$. We know there are constants $c_1, c_2$ for which $A_n=c_1\alpha^n+c_2\beta^n$, for all $n$.\\

By assumption $q(A_1)>0$, which implies $A_1<\alpha$ or $A_1>\beta$.\\

{\it Case 1.} $A_1<\alpha$. We will show that $\dfrac{A_{n+1}}{A_n}<\alpha$ for all $n$. This is equivalent to $c_1\alpha^{n+1}+c_2\beta^{n+1}<c_1\alpha^{n+1}+c_2\beta^n\alpha$. Simplifying we obtain $c_2 \beta^{n+1}<c_2\beta^n\alpha$. Since $\beta$ is positive, this simplifies to $c_2\beta<c_2\alpha$ or $0<c_2(\alpha-\beta)$. On the other hand $A_0=c_1+c_2=1$, and $A_1=c_1\alpha+c_2\beta<\alpha$, which implies $c_2\beta<(1-c_1)\alpha$ or $c_2(\beta-\alpha)<0$, which completes the proof for this case.\\

{\it Case 2.} $A_1>\beta$. We will show that $\dfrac{A_{n+1}}{A_n}>\beta$ for all $n$. This is equivalent to $c_1\alpha^{n+1}+c_2\beta^{n+1}>c_1\alpha^n\beta+c_2\beta^{n+1}$. Simplifying we obtain $c_1 \alpha^{n+1}>c_1\alpha^n\beta$, which is equivalent to $c_1\beta<c_1\alpha$. Note that $A_1=c_1\alpha+c_2\beta>\beta$ implies $c_1\alpha >(1-c_2)\beta$, which is equivalent to $c_1\alpha>c_1\beta$, as desired.

Now, assume $\alpha=\beta.$ Thus, $A_n=c_1\alpha^n+c_2n\alpha^n.$ Since $A_n=c_1\alpha^n+c_nn\alpha^n=(c_1+nc_2)\alpha^n$ is non-negative for all $n$ and $\alpha$ is positive we must have $c_1+nc_2>0$ for all $n$, which implies $c_2\geq 0.$ Note that $A_0=1=c_0$. Since $q(A_1)\neq 0$ we have $A_1\neq \alpha$ which means $c_1\alpha+c_2\alpha\neq \alpha$ or $c_2\neq 0.$ Thus, $c_2>0.$\\

Note that $\dfrac{A_{n+1}}{A_n}=\dfrac{c_1+(n+1)c_2}{c_1+nc_2}\alpha=\alpha+\dfrac{c_2\alpha}{c_1+nc_2}>\alpha$. This completes the proof.
\end{proof}

\section{Acknowledgements}
This paper and the research behind it would not have been possible without the support of my mentor, Dr. Roohollah Ebrahimian, Senior Lecturer, Department of Mathematics, University of Maryland, USA under whom I have worked for the past year. I would also like to thank my parents for continuously encouraging me and providing me with the resources I required.

\end{document}